\documentclass{elsarticle}

\usepackage{amssymb}
\usepackage{mathptmx}
\usepackage{amsmath}
\usepackage{amsthm}
\newtheorem{theorem}{Theorem}
\newtheorem{lemma}[theorem]{Lemma}
\newtheorem{corollary}[theorem]{Corollary}
\newtheorem{definition}[theorem]{Definition}
\usepackage{wrapfig}
\usepackage{graphicx}
\usepackage[utf8]{inputenc}
\usepackage{subfigure}
\begin{document}
\makeatletter
\def\ps@pprintTitle{%
	\let\@oddhead\@empty
	\let\@evenhead\@empty
	\let\@oddfoot\@empty
	\let\@evenfoot\@oddfoot
}
\makeatother
\begin{frontmatter}



\title{Optimal radio labeling for the Cartesian product of square mesh networks and stars}


\author{Linlin Cui\fnref{fn1}}
\ead{lin7232023@163.com}
\author{Feng Li\corref{cor1}\fnref{fn1}}
\ead{li2006369@126.com}
\cortext[cor1]{Corresponding author}
\affiliation[fn1]{organization={Computer College},
            addressline={Qinghai Normal University}, 
            city={Xi'ning},
            postcode={810000}, 
           state={Qinghai},
            country={China}}

\begin{abstract}
As the most critical component in the communication process, channels have a great impact on the communication quality of network. With the continuous expansion of network scale, the limited channel resources lead to the limitation of communication network scale. Therefore, achieving reasonable channel assignment and utilization becomes an extremely challenging problem. In order to solve this issue effectively, the channel assignment problem in communication networks can be transformed into a graph labeling problem, utilizing graphs to simulate the communication networks. In this paper, the topologies of mesh networks and stars are studied by constructing Cartesian product, and the lower bound and exact value of the optimal radio label of the Cartesian product of square mesh network and star $G=P(m,m)\Box K_{1,n}$ are obtained, where $m\geq 2$.
\end{abstract}




\end{frontmatter}


\section{Introduction}
\label{}
With the rapid advancement of wireless networks, wireless technology has become a prominent focus in global communications and the IT industry. Digital communication have driven the development of information technology, while the integration of wireless communication is completely transforming the way we communicate. In recent years, the emergence of various wireless technologies allowing people to connect to networks anytime and anywhere, fulfilling the dream of ubiquitous networking. However, the quality of communication network depends on channel assignment, and the limited channel resources make rational assignment strategy a key problem. In the face of complex vertex relationships in large-scale networks, graph theory methods have emerged. By simulating the network topology in graph theory, the channel assignment problem can be transformed into a graph labeling problem. As a branch of discrete mathematics, graph theory and graph labeling problems have successfully solved various real-world problems.

Analyzing the performance network topologies 
in wireless communication networks is very difficult.  Generally, graph theory is used to simulate the topology of the network. Simple graphs are used to build a large network, and the properties of the graph are used to analyze network performance. Graph theory serves as a natural framework for the precise mathematical treatment of complex networks, often consisting of nodes and connections. $N$ represents the number of stations in the system and $L$ represents the inter-node relationships.   In 2007, the definition of graph and more definitions of graph theory were marked in literature \cite{2007X}. Let the graph $G$ be an ordered binary group $G=(V, E)$, where $V$ represents the set of vertices, that is, the number of nodes $N$ and $E$ represents the set of edges, the number of interactions between nodes. If the number of edges $E = 0$ in a graph $G$, the graph $G$ is said to be a plain graph. The graphs studied in this paper are non-trivial graphs with $N\geq 2$ and $L\geq 1$.

In 1959, Sabidussi\cite{1959G} proposed the concept of Cartesian product for the construction of complex networks. Let $A=\{V,E\}$ and $B=\{N,L\}$ be two undirected graphs, and the Cartesian product of graph $A$ and graph $B$ is written $G=A\Box B$. If $G$ is still undirected, then its vertex set is $V(G)=V\times N$. In the graph $G$, take any two vertices $d,f$ and $o,p$, where $d,o\in V(A)$, $f,p\in V(B)$. If two vertices are adjacent, if and only if $d=o$, $fp\in E(B)$ or $f=p$, $do\in E(A)$, and graphs $A$ and $B$ are called factor graphs of $G$.

The Cartesian product method constructs large graphs from several non-trivial simple graphs.  Large graphs have the properties of simple graphs and are homeomorphic to simple trivial graphs. Let $G$ have a total of $n$ vertices, divide $n$ vertices into two independent parts, the vertices in the same part are not adjacent to each other, and each vertex in different parts is adjacent, such a graph is called a complete bipartite graph, star is a special complete bipartite graph, its part contains only one vertex, as the center vertex, and the rest of the $n-1$ vertices exist as leaf vertices. It has all the properties of a complete bipartite graph and is denoted $K_{1,n}$. Let $G$ have a total of $m$ vertices, and is the set of vertices and edges that do not coincide from one vertex to another on the graph. Such a graph is called a path, denoted $P_m$, and the path length is $m-1$. The topology of a mesh network is the Cartesian product of two undirected paths $P_l\Box P_m$, denoted as $P(l,m)$. If two undirected paths of the same order are constructed at the same time, denoted as a square mesh network, denoted as $P(m,m)$. The network topology constructed by the Cartesian product of square mesh network and star has the properties of both star network and mesh network, and the network topology is complex.

Channel is the most important part of network communication, using channels to assign different frequencies to vertices, so as to realize the transmission and exchange of data between communication networks. In order to solve the problem of signal interference in network communication. Hale\cite{1980F} proposed the channel assignment problem for the first time in 1980. To better solve the problem of data tampering or loss caused by interference in network communication, which affects call quality. Griggs and Yeh et al. \cite{1992G} proposed the 2-distance labeling method, that is, the distance between two adjacent nodes should be at least 2. That is to say, when the data is transmitted or exchanged between two sites in the communication process, the frequency of data transmission must be separated by at least two, and the degree of interference in the communication is related to its geographical location. If the distance between the two stations is closer, the degree of interference is greater.

Subsequently, Chartrand et al. \cite{2001R} applied the channel assignment problem to FM broadcast TV stations and used different channel assignment algorithms to solve the channel assignment problem, such as the minimum collision algorithm and the minimum margin algorithm. These algorithms can help TV channels assign channel efficiently and reduce interference in the case of limited spectrum resources. The channel assignment problem is transformed into graph labeling problem to improve signal interference in communication networks. Graph labeling advances the data science and information age in confidential data management, communication encryption, channel assignment, and algorithm growth. Using the graph vertex labeling problem to solve the channel assignment problem, this extends the definition of radio labeling, which is a mapping function:
\begin{equation*}
    \rho:V(Q)\rightarrow Z^+\cup{0},
\end{equation*}

Such that

\begin{equation*}
   |\rho(a)-\rho(b)|\geq diam(Q)+1-d(a,b),\ \forall a,b\in V(Q). 
\end{equation*}

Where the maximum value of the multilevel distance label of $|\rho(a)-\rho(b)|$ is the span of the graph $Q$, denoted $span(\rho)=max|\rho(a)-\rho(b)|$. The smallest possible span of a graph $Q$ is called its radio number, denoted $rn(Q)=min\{span(\rho)\}$. $\rho(a)$ is the radio label of vertex $a$, the minimum distance between two adjacent vertices in a graph $Q$ is denoted $d(a,b)$, and the maximum of the minimum distance between any two vertices is the diameter of the graph $Q$ and is denoted $diam(Q)$.

Radio label, also known as multilevel label, is used to identify and distinguish different radio stations in radio communication. Each radio station has a unique label, usually composed of letters and numbers. These labels can also be represented by radio coloring numbers. Radio labels can represent radio equipment in a specific frequency range, and radio coloring numbers can indicate the corresponding colors of those frequency ranges. The purpose of radio labeling is to ensure the accuracy and reliability of communication.

However, determining the radio label of the graph is a very challenging problem. The topology of graphs can be divided into special graphs and general graphs. Special graphs have unique properties, while the topology of general graphs is more random. Therefore, most research focuses on special graphs, and the radio labels of some of these graphs have been determined. In 2021, scholars \cite{2021R} determined the radio number of a class of commutative rings using zero-divisor graph; Badr et al. \cite{2020U} proposed an improved upper bound of radio $k$-coloring of a given graph and another graph, determined the radio $k$-coloring problem of path graphs, and proposed a new model. Kim et al. \cite{2015R} determined the radio number of the Cartesian product of a complete graph and a path; Palani et al. \cite{2022r} identified radio labels for some splitting graphs; Bantva and Liu in \cite{2021O,2024R}, determined the optimal radio labels of the block graph and line graph of the tree, as well as the lower bound of the radio number of the Cartesian product of two trees, and the radio number of the Cartesian product of two stars and a path and a star. Korz et al. \cite{2022N} considered radio $k$-labelings for three families of range graphs and proposed some improved theoretical lower and upper limits for the number of radio $k$-labelings. Bloomfield and Liu et al. \cite{2021C} combined a lower bound approach with cyclic group structure to determine the value of $rnk(C_n)$, Partial results are obtained when $n$ and $k$ have different parity. Liu and Saha et al. \cite{2021I} proposed some improved lower bounds of trees based on the existing lower bounds of radio labels associated with trees. Some radio labels for other special graphs and for other special product graphs have been identified, see reference \cite{2023L,2020O,2009P,2005M,2011T,2024Y,2024W,2023H,2011S,2012L,2022A,2019Z,2021A}.

The radio label values of some special graphs or special product graphs have been determined, but the radio labels for constructing larger network topologies using mesh networks are less studied. In this paper, the Cartesian product $P(m,m)\Box K_{1,n}$ of the topology of square mesh and star network is studied by means of Cartesian product construction method, where $m\geq 2$, the lower bound and the exact value of the optimal radio label are obtained.

\section{Main Results}
To help readers better understand, Table 1 provides some symbolic explanations used in this paper.
\begin{table}[htpb]
\caption{Symbol explanation.}\label{tab1}
\begin{tabular}{c|l}
symbolic &  Description \\
\hline
$P_m$ &  A path of order $m$, that is, a path with $m$ vertices \\
$K_{1,n}$ &  A star with $n+1$ vertices \\
$P(m,m)$ & Square mesh network, that is, the Cartesian product graph of two $m$-order paths \\
$d(a,b)$ & The shortest distance between any two vertices $a$ and $b$\\
$rn(G)$ & The radio number of graph $G$ \\
$diam(G)$ & The diameter of graph $G$ \\
$f(v)$ & The radio label of vertex $v$ \\
\end{tabular}
\end{table}

\begin{definition}
   Let $G=P(m,m)\Box K_{1,n}$ be the Cartesian product of an $m$-order square mesh network and a star with $n+1$ vertices, where $m\geq 2$, then the graph $G$ has a total of ($n+1)\times m^2$ vertices.  
\end{definition}

\begin{lemma} 
    Let $H_i$ be a connected graph, and the diameter of a graph is given by the Cartesian product of $\omega$ $H_i$ graphs and the order of the product factor graph $H_i$ is $\omega_i$
\end{lemma}
\begin{equation}
    diam(G)=diam(H_1+H_2+... +H_\omega).
\end{equation}
\begin{corollary}
    Let $G=P(m,m)\Box K_{1,n}$ be the Cartesian product graph of an $m$-order square mesh network and a star with $n+1$ vertices, then the diameter of the graph $G$ is $diam(G)=2m$.
\end{corollary}

\begin{proof}
    Let $P_m$ be a path of length $m-1$ composed of $m$ vertices connected sequentially. The minimum distance between any two vertices $a$ and $b$ in $P_m$ is denoted as $d(a,b)$, and its diameter is $max\{d(a,b)\}$. Then, according to the definition of path, the longest path with the minimum distance between any two vertices in $P_m$ is $m-1$, so $diam(P_m)=m-1$. Since the topology of the square mesh network is the Cartesian product of two paths of the same order, as shown in Fig 1, the diameter of the square mesh network is $diam(P(m,m))=2\times(m-1)$.
    
In a complete bipartite graph, vertices can be divided into two independent parts, each vertex in the first part is connected to the vertex in the second part. But each vertex in the same part is not adjacent to the other vertices. As a special type of complete bipartite graph, a star graph has a single central vertex in the first set that is adjacent to n vertices in the second set. By the definition of a star, any two vertices have a maximum path of 2 when they are in the same part of the graph. so $diam(K_{1,n})=2$.

According to Corollary 3, $diam(P(m,m)\Box K_{1,n})=2\times(m-1)+2=2m$.
\end{proof}

\begin{figure}[htpb]
	\centerline{\includegraphics{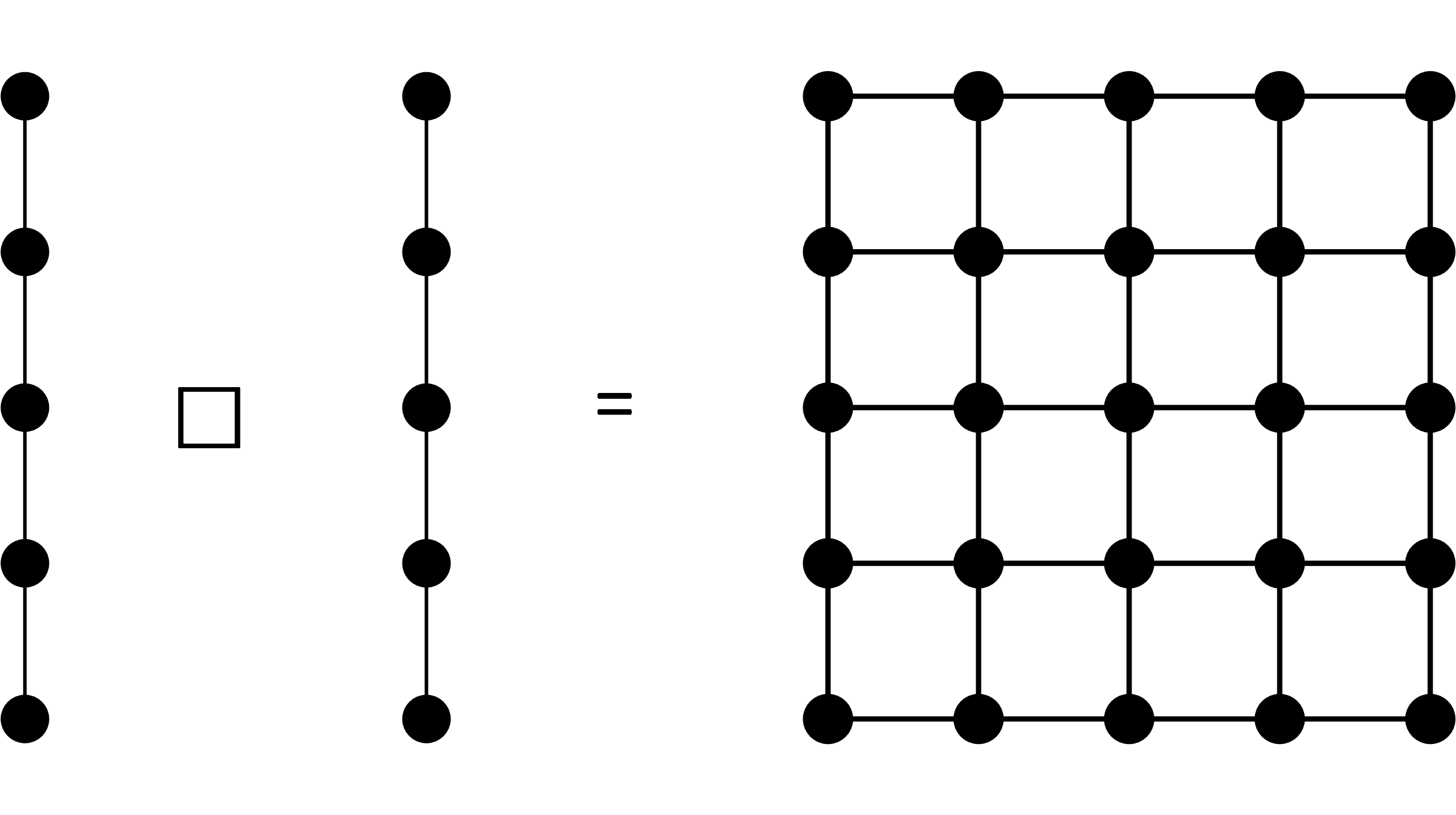}}
 \vspace{-2mm}
	\caption{Topology construction of 5-order square mesh network}
	\label{fig}
\end{figure}
    \begin{definition}
        Let $t(i)\subset P(m,m)\Box K_{1,n}$ be a subgraph of  $G$, where $i\in [1,m^2]$ and $m$ is even. Take any two subgraphs $t(j)$ and $t(j+\frac{m^2}{2})$ in $G$ to induce the subgraph as $G^{\prime}(j)$, ($j\in[1,\frac{m^2}{2}]$), then $G$ contains $\frac{m^2}{2}$ $G^{\prime}(j)$ subgraphs.
    \end{definition}

For a clearer understanding of Definition 4, the topology of $P(6,6)\Box K_{1,4}$ is constructed as shown in Fig 2, which contains a total of 18 $G^{\prime}(j)$ subgraphs.
\begin{corollary}
    Let $V_j(1)$ be the central vertex of $G^{\prime}(j)$, where $j\in[1,\frac{m^2}{2}]$ and $m$ is even, then the radio number of $G^{\prime}(j)$ is $rn(G^{\prime}(j))\geq\frac{3m}{2}+2+mn+n$.
\end{corollary}
\begin{proof}
     Let $G^{\prime}(j)$ be the subgraph induced by $t(j)$ and $t(j+\frac{m^2}{2})$, $V_j(1)$ be the central vertex of $G^{\prime}(j)$, take any two vertices $u,v$ in $G^{\prime}(j)$, where $u\in t(j)$, $v\in t(j+\frac{m^2}{2})$, then
     
     \begin{equation}
	d(u,v)=\left\{
	\begin{aligned}
		&\frac{m}{2},if \,\, u,v\in V_j(1);\\
		&m-1,if \,\, u\in V_j(1)\,or\,v\in V_j(1);\\
		&m,if \,\,u,v\notin V_j(1).
	\end{aligned}
	\right.
\end{equation}

\begin{figure}[htpb]
	\centerline{\includegraphics{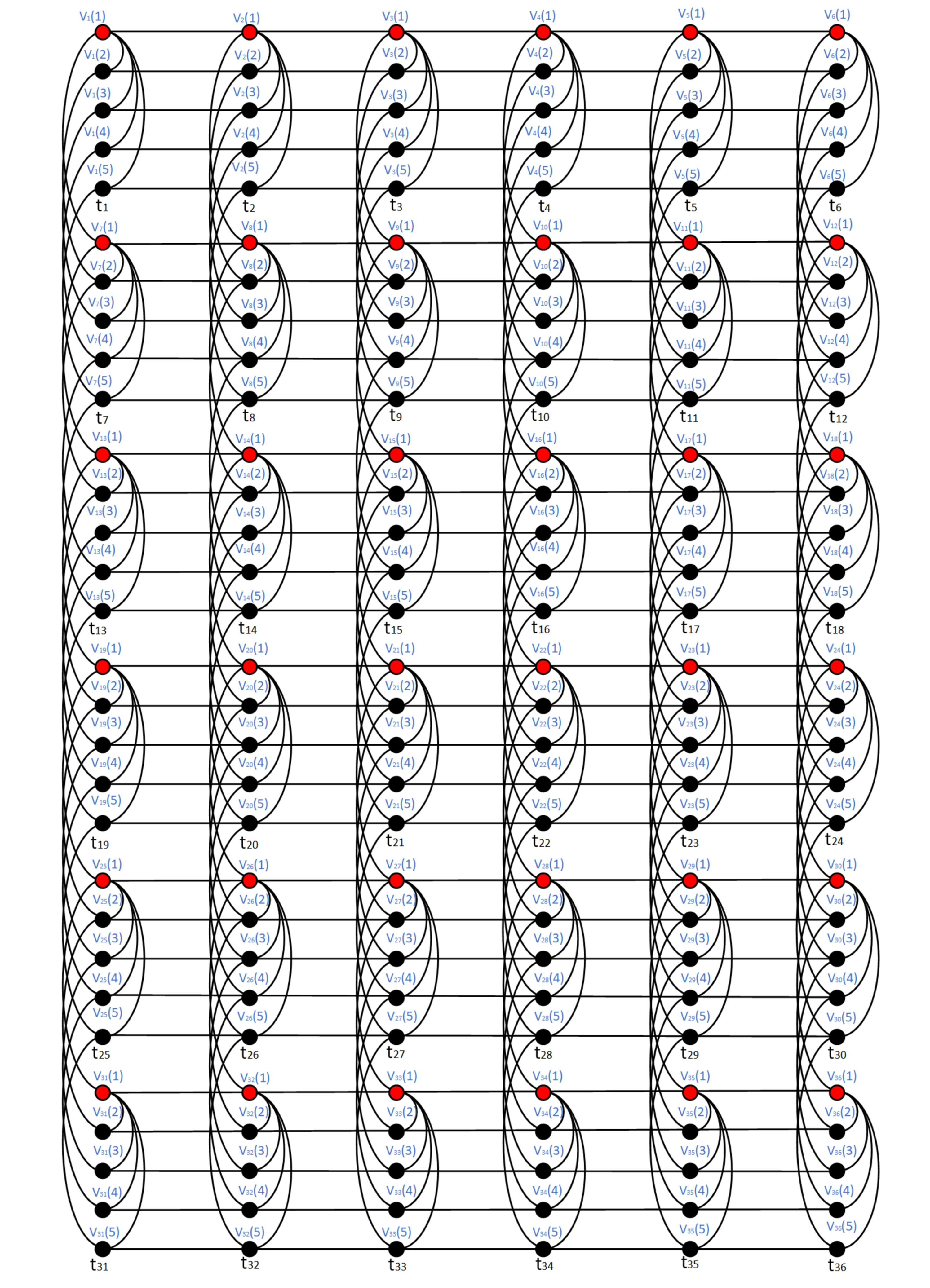}}
	\caption{The Cartesian product graph $P(6,6)\Box K_{1,4}$ of a star with $5$ vertices and a $6$-order square mesh network, where the central vertex of each subgraph $t(i)$ is marked in red}
	\label{fig}
\end{figure}

Let $V_1(1)$ be the central vertex of t(1), $V_{1+\frac{m^2}{2}}(1)$ be the central vertex of $t(1+\frac{m^2}{2})$, and $f(V_1(1))=0$. Since $V_1(1)$ and $V_{1+\frac{m^2}{2}}(1)$ are the center vertices of the two subgraphs, its shortest path is obtained according to equation (2)
\begin{equation}
    \begin{split}
        d(V_1(1),V_{1+\frac{m^2}{2}}(1))=\frac{m}{2},
    \end{split}
\end{equation}

According to the definition of graph $G$ diameter and radio label obtained in Corollary 2, it can be obtained
\begin{equation}
    \begin{split}
        f(V_{1+\frac{m^2}{2}}(1))&\geq f(V_1(1))+diam(G)+1-d(V_1(1),V_{1+\frac{m^2}{2}}(1))\\
&=\frac{3m}{2}+1.
    \end{split}
\end{equation}

Similarly, since $V_{1+\frac{m^2}{2}}(1)$ is the central vertex of the subgraph $t(1+\frac{m^2}{2})$, it follows from the equation (2) can be obtained
\begin{equation}
    d(V_{1+\frac{m^2}{2}}(1),V_1(2))
=m-1,
\end{equation}

According to the definition of radio label, the multilevel distance label of vertex $V_1(2)$ can be obtained as
\begin{equation}
    \begin{split}
        f(V_1(2))&\geq f(V_{1+\frac{m^2}{2}}(1))+diam(G)+1-d(V_{1+\frac{m^2}{2}}(1),V_1(2))\\
&=\frac{5m}{2}+3.
    \end{split}
\end{equation}

Since vertices $V_1(2)$ and $V_{1+\frac{m^2}{2}}(3)$ are leaf vertices of the subgraph, the equation (2) can be obtained
\begin{equation}
 d(V_1(2),V_{1+\frac{m^2}{2}}(3))=m,   
\end{equation}

According to the definition of radio label, the multilevel distance label of vertex $V_{1+\frac{m^2}{2}}(3)$ can be obtained as
\begin{equation}
    \begin{split}
        f(V_{1+\frac{m^2}{2}}(3))&\geq f(V_1(2))+diam(G)+1-d(V_1(2),V_{1+\frac{m^2}{2}}(3))\\
&=\frac{5m}{2}+3+m+1.
    \end{split}
\end{equation}

After repeated iterations, it can be obtained that the last two vertices in the subgraph are leaf vertices, it follows from the equation (2) can be obtained
\begin{equation}
 d(V_1(n),V_{1+\frac{m^2}{2}}(n+1))
=m,   
\end{equation}

Then, the radio label of the subgraph $G^{\prime}(j)$ is
\begin{equation}
    \begin{split}
        f(V_{1+\frac{m^2}{2}}(n+1))&\geq f(V_1(n))+diam(G)+1-d(V_1(n),V_{1+\frac{m^2}{2}}(n+1))\\
&=\frac{5m}{2}+3+(n-1)\times(m+1)\\
&=\frac{3m}{2}+2+mn+n.
    \end{split}
\end{equation}

The radio number of $G^{\prime}(j)$ is $rn(G^{\prime}(j))\geq\frac{3m}{2}+2+mn+n$.
\end{proof}
\begin{theorem}
    Let $G=P(m,m)\Box K_{1,n}$ be the Cartesian product of an $m$-order square network and a star with $n+1$ vertices, where $m\geq 2$ is even, then $rn(G)\geq\frac{{3m}^3}{4}+\frac{{2m}^2+m^3n+m^2n}{2}$.
\end{theorem}
 \begin{proof}
      Let $G^{\prime}(j)\subset G$ be a subgraph of $G$, as shown in Fig 2, and it follows from Definition 4 that $G$ contains $\frac{m^2}{2}$ $G^{\prime}(j)$ subgraphs. The radio number of $G^{\prime}(j)$ obtained by Corollary 5 is $rn(G^{\prime}(j))\geq\frac{3m}{2}+2+mn+n$. Therefore, the radio number of $G$ can be obtained as
      \begin{equation}
          \begin{split}
              rn(G)&\geq\sum_{j=1}^{\frac{m^2}{2}}{rn(G^{\prime}(j))}\\
&=\frac{m^2}{2}\times\left(\frac{3m}{2}+2+mn+n\right)\\
&=\frac{{3m}^3}{4}+\frac{{2m}^2+m^3n+m^2n}{2}.
          \end{split}
      \end{equation}

      When $m$ is even, the radio number of $G$ is $rn(G)\geq\frac{{3m}^3}{4}+\frac{{2m}^2+m^3n+m^2n}{2}$.
 \end{proof}  
 \begin{definition}
     Let $t(i)\subset P(m,m)\Box K_{1,n}$ be a subgraph of $G$, where $i\in[1,m^2]$, $m$ is odd. The subgraph induced by taking any two subgraphs $t(x)$ and $t(x+\frac{m(m-1)}{2})$ in $G$ is $G^{\prime\prime}(x), (x\in[1,\frac{m(m-1)}{2}])$. Then $G$ contains $\frac{m(m-1)}{2}$ $G^{\prime\prime}(x)$ subgraphs.
 \end{definition}

 For a clearer understanding of Definition 7, construct the topology of $P(7,7)\Box K_{1,4}$ as shown in Fig 3.
 \begin{corollary}
     Let $V_{x}(1)$ be the central vertex of $G^{\prime\prime}(x)$, where $x\in[1,\frac{m(m-1)}{2}]$, $m$ is odd. The radio number of $G^{\prime\prime}(x)$ is $rn(G^{\prime\prime}(x))\geq\frac{3mn-n+2}{2}$.
 \end{corollary}
    \begin{proof}
        Let $G^{\prime\prime}(x)$ be the subgraph induced by $t(x)$ and $t(x+\frac{m(m-1)}{2})$, and $V_x(1)$ be the central vertex of $G^{\prime\prime}(x)$. Take any two vertices $a,b$ in $G^{\prime\prime}(x)$, where $a\in t(x)$, $b\in t(x+\frac{m(m-1)}{2})$, then
        \begin{equation}
	d(a,b)=\left\{
	\begin{aligned}
		&\frac{m}{2}-1,if \,\, a,b\in V_x(1);\\
		&\frac{m}{2}+1,if \,\, a\in V_x(1)\,or\,b\in V_x(1);\\
		&\frac{m+3}{2},if \,\, a,b\notin V_x(1).	
	\end{aligned}
	\right.
\end{equation}

Let $V_1(1)$ be the central vertex of $t(1)$, $V_{1+\frac{m(m-1)}{2}}(1)$ be the central vertex of $t(1+\frac{m(m-1)}{2})$, and $f(V_1(1))=0$. According to equation (12), it can be obtained
\begin{equation}
    d(V_1(1),V_{1+\frac{m(m-1)}{2}}(1))
=\frac{m-1}{2},
\end{equation}
\begin{figure}[htpb]
	\centerline{\includegraphics{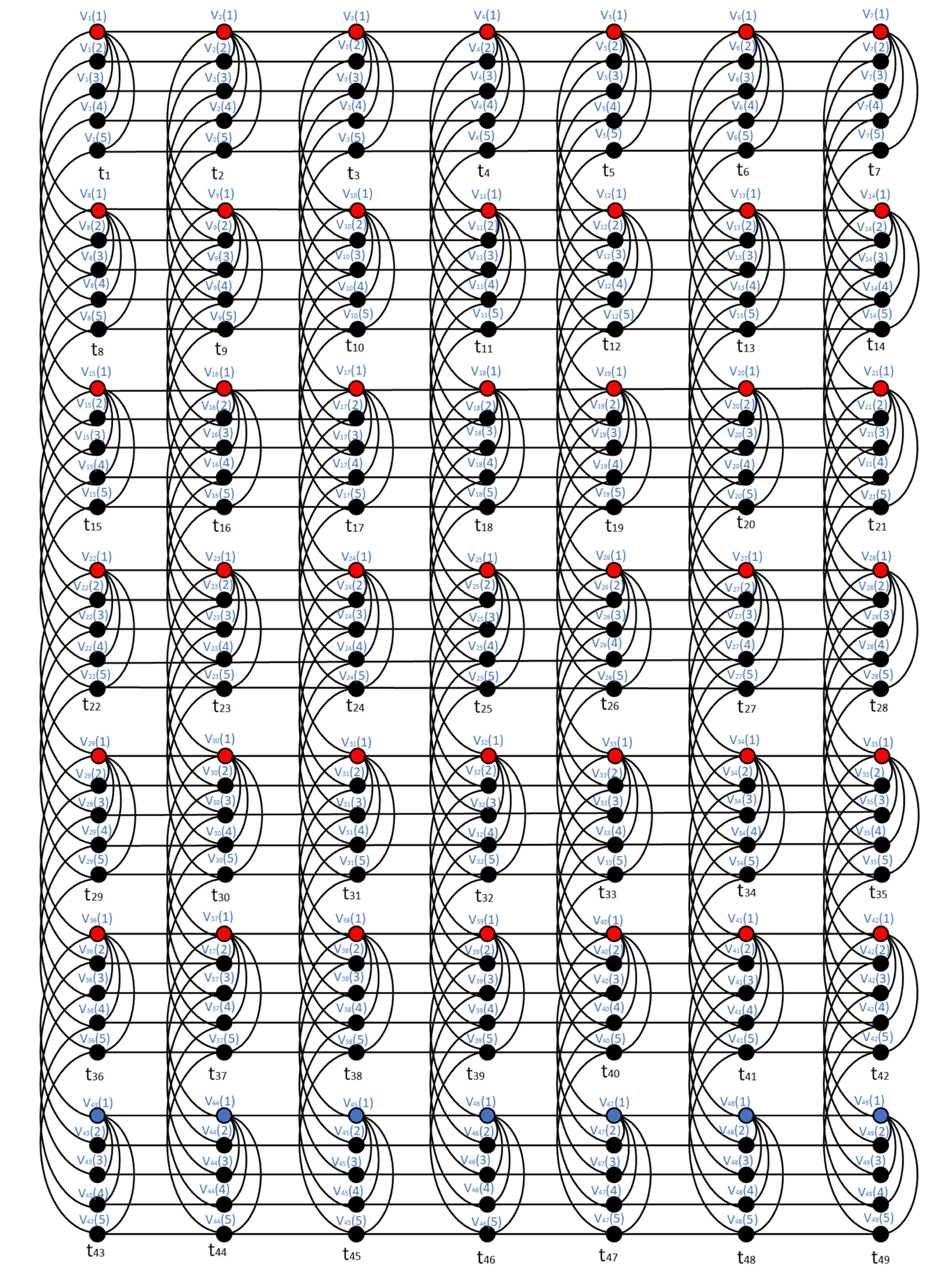}}
	\caption{The Cartesian product graph $P(5,5)\Box K_{1,4}$ of a $5$-order square mesh network and a star with $5$ verticesis. It divided into two parts and its radio label is calculated separately, denoted by red and blue marks respectively, and the vertices marked by the two colors are the center vertices of the subgraph $t(i)$}
	\label{fig}
\end{figure}
According to the definition of radio label, the multilevel distance label of $V_{1+\frac{m(m-1)}{2}}(1)$ can be obtained as
\begin{equation}
    \begin{split}
        f(V_{1+\frac{m(m-1)}{2}}(1))&\geq f(V_1(1))+diam(G)+1-d(V_1(1),V_{1+\frac{m^2}{2}}(1))\\
&=\frac{m+3}{2}+m.
    \end{split}
\end{equation}

Since $V_{1+\frac{m(m-1)}{2}}(1)$ is the central vertex of the subgraph $t(1+\frac{m(m-1)}{2})$,  according to equation (12), we can obtain
\begin{equation}
    d(V_{1+\frac{m(m-1)}{2}}(1),V_1(2))
=\frac{m+1}{2},
\end{equation}

Then the multilevel distance label of vertex $V_1(2)$ is
\begin{equation}
    \begin{split}
        f(V_1(2))&\geq f(V_{1+\frac{m(m-1)}{2}}(1))+diam(G)+1-d(V_{1+\frac{m(m-1)}{2}}(1),V_1(2))\\
&=\frac{m+3}{2}+m+\frac{m+1}{2}.
    \end{split}
\end{equation}

Since $V_1(2)$ and $V_{1+\frac{m(m-1)}{2}}(3)$ are leaf vertices, according to equation (12), it can be obtained
\begin{equation}
    d(V_1(2),V_{1+\frac{m(m-1)}{2}}(3))
=\frac{m+3}{2},
\end{equation}

According to the definition of radio label, the multilevel distance label of $V_{1+\frac{m(m-1)}{2}}(3)$ can be obtained as
\begin{equation}
    \begin{split}
        f(V_{1+\frac{m(m-1)}{2}}(3))&\geq f(V_1(2))+diam(G)+1-d(V_1(2),V_{1+\frac{m(m-1)}{2}}(3))\\
&=3m+1+\frac{m+1}{2}.
    \end{split}
\end{equation}

Similarly, the two vertices are for the graph of the leaf vertices, and according to the equation (12) can be obtained
\begin{equation}
    d(V_{1+\frac{m(m-1)}{2}}(3),V_1(4))
=\frac{m+3}{2},
\end{equation}

Then the multilevel distance label of $V_1(4)$ is obtained as
\begin{equation}
    \begin{split}
        f(V_1(4))&\geq f(V_{1+\frac{m(m-1)}{2}}(3))+diam(G)+1-d(V_{1+\frac{m(m-1)}{2}}(3),V_1(4))\\
&=4m+1+\frac{m+1}{2}-\frac{m-1}{2}.
    \end{split}
\end{equation}

After continuous iteration can be obtained, the last two vertices of the subgraph are leaf vertices, then
\begin{equation}
    d(V_1(n),V_{1+\frac{m(m-1)}{2}}(n+1))
=\frac{m+3}{2},
\end{equation}

According to the definition of radio label, the maximum radio label of the subgraph is
\begin{equation}
    \begin{split}
        f(V_{1+\frac{m(m-1)}{2}}(n+1))&\geq f(V_1(n))+diam(G)+1-d(V_1(n),V_{1+\frac{m(m-1)}{2}}t(n+1))\\
&=mn+1+\frac{m+1}{2}+(n-1)\times\frac{m-1}{2}\\
&=\frac{3mn-n+2}{2}.
    \end{split}
\end{equation}

The radio number of $G^{\prime\prime}(x)$ is $rn(G^{\prime\prime}(x))\geq\frac{3mn-n+2}{2}$.
    \end{proof} 
\begin{theorem}
    Let $G^{\prime\prime}(x)\subseteq G^{\prime}(\ast)$ be a subgraph of $G$, then $rn(G^{\prime}(\ast))=m^2-m^2n-m+\frac{m^3n+mn}{4}$.
\end{theorem}
\begin{proof}
    Let $G^{\prime}(\ast)$ be a subgraph of $G$. It follows from Definition 7 that $G$ contains $\frac{m(m-1)}{2}$ $G^{\prime\prime}(x)$ subgraphs. From the result of Corollary 8, we get $rn(G^{\prime\prime}(x))\geq\frac{3mn-n+2}{2}$. Without loss of generality, it follows that the radio number of $G^{\prime}(\ast)$ is
    \begin{equation}
        \begin{split}
            rn(G^{\prime}(\ast))&\geq\sum_{x=1}^{\frac{m(m-1)}{2}}{G^{\prime\prime}(x)}\\
            &=\frac{m(m-1)}{2}\times(\frac{3mn-n+2}{2})\\
&=\frac{m^3n-4m^2n+4m^2+mn-4m}{4}.
        \end{split}
    \end{equation}

    The radio number of $G^{\prime}(\ast)$ is $rn(G^{\prime}(\ast))\geq\frac{3m^3n-4m^2n+4m^2+mn-4m}{4}$.
\end{proof}
\begin{definition}
    Let $G^{\prime\prime}(\ast)$ be a subgraph induced by $t(i+m(m-1))$, where $i\in[1,m]$. In $G^{\prime\prime}(\ast)$, take the subgraphs $t(1+m(m-1))$, $t(\frac{m+1}{2}+m(m-1))$ and $t(m+m(m-1))$ as $G^{\prime}(\ast\ast)$. And take $V_{1+m(m-1)}(1)$, $V_{\frac{m+1}{2}+m(m-1)}(1)$, $V_{m+m(m-1)}(1)$ and $V_ {m+m(m-1)}(1)$ as the central vertices of $t(1+m(m-1))$, $t(\frac{m+1}{2}+m(m-1))$ and $t(m+m(m-1))$, respectively.
\end{definition}
\begin{corollary}
    Let $P^\prime(j)$ be a class of paths in $G$, $j\in[1,n+1]$, and $P^\prime(j)=v_a(1)\stackrel{\alpha}{\longrightarrow}v_b(\frac{m+1}{2})\stackrel{\beta}{\longrightarrow}v_c(m)$ such that $a\neq b\neq c$, $V_a\in V_1(1+m(m-1))$, $V_b(\frac{m+1}{2})\in V(\frac{m+1}{2}+m(m-1))$, $V_c(m)\in V(m+m(m-1))$, and $1\le a,b,c\le n+1$. It can be verified that $P^\prime(j)$ contains three cases, defined as follows without loss of generality:
\end{corollary}
\begin{equation*}
	\begin{split}
		P'_1(j)=&\{V_{1+m(m-1)}(1)\stackrel{\frac{m+1}{2}}{\longrightarrow}V_{\frac{m+1}{2}+m(m-1)}(3)\stackrel{\frac{m+3}{2}}{\longrightarrow}
		V_{m+m(m-1)}(2),V_{1+m(m-1)}(3)\\&\stackrel{\frac{m+3}{2}}{\longrightarrow}V_{1+m(m-1)}(2)\stackrel{\frac{m+1}{2}}{\longrightarrow}V_{m+m(m-1)}(1)\}; \\
		P'_2(j)=&\{V_{1+m(m-1)}(2)\stackrel{\frac{m+1}{2}}{\longrightarrow}V_{\frac{m+1}{2}+m(m-1)}(1)\stackrel{\frac{m+1}{2}}{\longrightarrow}V_{m+m(m-1)}(3)\};\\
		P'_3(j)=&\{V_{1+m(m-1)}(x)\stackrel{\frac{m+3}{2}}{\longrightarrow}V_{(\frac{m+1}{2})+m(m-1)}(y)\stackrel{\frac{m+3}{2}}{\longrightarrow}V_{m+(m-1)}(z), 
		4\le x,y,z\\&\le n+1\}.
	\end{split}
\end{equation*}
\begin{theorem}
    Let $f$ be the radio label of $G=P(m,m)\Box K_{1,n}$, and $m$ is odd, in a class of paths of $P(j)\subseteq P^{\prime}(j)$. Take the $t (\frac{m+1}{2}(m-1)+m)$ of the vertex set $V(t(\frac{m+1}{2}(m-1)+m))=\{V_{\frac{m+1}{2}+m(m-1)}(1), V_ {\frac{m+1} {2}+m(m-1)}(2),V_{\frac{m+1}{2}+m(m-1)}(3),{V_{\frac{m+1}{2}+m(m-1)}(d)\}}, 4\le d\le\ n+1$, $j\in(1,n+1)$. Now take $V(t(\frac{m+1}{2}(m-1)+m))$ of a vertex $V$, if $f$ is the largest radio labels, then
    \begin{equation*}
	f(v)=\left\{
	\begin{aligned}
		&\frac{m}{2}+2,if \,\, v\in \{V_h(\frac{m+1}{2}+m(m-1)), 1\le h\le 3\};\\
		&\frac{m-1}{2}+2m,if \,\, others.
	\end{aligned}
	\right.	
\end{equation*}
\end{theorem}
\begin{proof}
    Let $P(j)\subset P(m,m)\Box K_{1,n}$, then the radio number of any vertex on $V(P(j))$ is based on $diam(G)$. Next, we discuss the three cases separately.
    
Case 1(a) :For $P'_1(j)=V_{1+m(m-1)}(1)\stackrel{\frac{m+1}{2}}{\longrightarrow}V_{\frac{m+1}{2}+m(m-1)}(3)\stackrel{\frac{m+3}{2}}{\longrightarrow}V_{m+m(m-1)}(2)$. Let $V_{1+m(m-1)}(1)$ be the central vertex of $t(1+m(m-1))$ and satisfy $f(V_{1+m(m-1)}$ $(1))=0$. According to the topology of the Cartesian product graph of Fig.3, it can be obtained
\begin{equation}
    d(V_{1+m(m-1)}(1),V_{m+m(m-1)}(2))
=m,
\end{equation}

Then according to the definition of radio label, the multilevel distance label of $V_{m+m(m-1)}(2)$ can be obtained as
\begin{equation}
    \begin{split}
        f(V_{m+m(m-1)}(2))&\geq f(V_{1+m(m-1)}(1))+diam(G)+1-d(V_{1+m(m-1)}(1),V_{m+m(m-1)}(2))\\
&=m+1.
    \end{split}
\end{equation}

Similarly, the distance between the vertices of $P^{\prime}_{1}(j)$ in the graph is obtained
\begin{equation}
    d(V_{m+m(m-1)}(2),V_{\frac{m+1}{2}+m(m-1)}(3))
=\frac{m+3}{2},
\end{equation}

By the definition of radio label, the maximum radio label of the first path in path $P^{\prime}_{1}(j)$ can be obtained as
\begin{equation}
    \begin{split}
        f(V_{\frac{m+1}{2}+m(m-1)}(3))&\geq f(V_{m+m(m-1)}(2))+diam(G)+1\\&\quad
-d(V_{m+m(m-1)}(2),V_{\frac{m+1}{2}+m(m-1)}(3))\\
&=2m+\frac{m+1}{2}.
    \end{split}
\end{equation}

Case 1(b) :For ${P^\prime}_1(j)=V_{1+m(m-1)}(3)\stackrel{\frac{m+3}{2}}{\longrightarrow}V_{1+m(m-1)}(2)\stackrel{\frac{m+1}{2}}{\longrightarrow}V_{m+m(m-1)}(1)$. Let $V_{m+m(m-1)}(1)$ be the central vertex of $t(m+m(m-1))$ and satisfy $f(V_{m+m(m-1)}(1))=0$. According to the topology of Fig.3, it can be obtained
\begin{equation}
    d(V_{m+m(m-1)}(1),V_{1+m(m-1)}(3))
=m,
\end{equation}

The radio label of $V_{1+m(m-1)}(3)$ is
\begin{equation}
    \begin{split}
        f(V_{1+m(m-1)}(3))&\geq f(V_{m+m(m-1)}(1))+diam(G)+1\\&\quad-d(V_{m+m(m-1)}(1),V_{1+m(m-1)}(3))\\
&=m+1.
    \end{split}
\end{equation}

Similarly, the distance between paths of path $P^{\prime}_{1}(j)$ in the product graph can be obtained
\begin{equation}
    \begin{split}
        d(V_{1+m(m-1)}(3),V_{\frac{m+1}{2}+m(m-1)}(2))
=\frac{m+3}{2},
    \end{split}
\end{equation}

Then the maximum radio label of the second path in $P^{\prime}_{1}(j)$ is
\begin{equation}
    \begin{split}
        f(V_{\frac{m+1}{2}+m(m-1)}(2))&\geq f(V_{1+m(m-1)}(3))+diam(G)+1\\&\quad-d(V_{1+m(m-1)}(3),V_{\frac{m+1}{2}+m(m-1)}(2))\\
&=2m+\frac{m+1}{2}.
    \end{split}
\end{equation}

Case 2 :For $P'_2(j)=V_{1+m(m-1)}(2)\stackrel{\frac{m+1}{2}}{\longrightarrow}V_{\frac{m+1}{2}+m(m-1)}(1)\stackrel{\frac{m+1}{2}}{\longrightarrow}V_{m+m(m-1)}(3)$. Let  $V_{\frac{m+1}{2}+m(m-1)}(1)$ be the central vertex of $t(\frac{m+1}{2}+m(m-1))$ and satisfy $f(V_{1+m(m-1)}(2))$ $=0$. According to the product graph can be obtained
\begin{equation}
    d(V_{1+m(m-1)}(2),V_{m+m(m-1)}(3))
=m+1,
\end{equation}

The radio label of $V_{m+m(m-1)}(3)$ in the second type path $P^{\prime}_{2}(j)$ is
\begin{equation}
    \begin{split}
        f(V_{m+m(m-1)}(3))&\geq f(V_{1+m(m-1)}(2))+diam(G)+1\\&\quad-d(V_{m+m(m-1)}(2),V_{1+m(m-1)}(3))\\
&=m.
    \end{split}
\end{equation}

Similarly, the distance between paths of the second type path $P^{\prime}_{2}(j)$ in the product graph can be obtained
\begin{equation}
    d(V_{m+m(m-1)}(3),V_{\frac{m+1}{2}+m(m-1)}(1))
=\frac{m+1}{2},
\end{equation}

Then the maximal radio label of the second type of path $P^{\prime}_{2}(j)$ is obtained as
\begin{equation}
    \begin{split}
        f(V_{\frac{m+1}{2}+m(m-1}(1))&\geq f(V_{m+m(m-1)}(3))+diam(G)+1\\&\quad-d(V_{m+m(m-1)}(3),V_{\frac{m+1}{2}+m(m-1)}(1))\\
&=2m+\frac{m+1}{2}.
    \end{split}
\end{equation}

Case 3 :For $P'_3(j)=V_{1+m(m-1)}(x)\stackrel{\frac{m+3}{2}}{\longrightarrow}V_{(\frac{m+1}{2})+m(m-1)}(y)\stackrel{\frac{m+3}{2}}{\longrightarrow}V_{m+(m-1)}(z),$ $4\le x,y,z\le n+1$. Let $f(V_{1+m(m-1)}(x))=0$, and all vertices are leaf vertices of the star. By the graph topology structure can be obtained 
\begin{equation}
    d(V_{1+m(m-1)}(x),V_{m+m(m-1)}(z))
=m+1,
\end{equation}

According to the definition of radio label, the multilevel distance label of $V_{m+m(m-1)}(z)$ can be obtained as
\begin{equation}
    \begin{split}
        f(V_{m+m(m-1)}(z))&\geq f(V_{1+m(m-1)}(x))+diam(G)+1\\&\quad-d(V_{1+m(m-1)}(x),V_{m+m(m-1)}(z))\\
&=m.
    \end{split}
\end{equation}

Similarly, the distance between paths of the third type $P^{\prime}_{3}(j)$ is obtained
\begin{equation}
    d(V_{m+m(m-1)}(z),V_{\frac{m+1}{2}+m(m-1)}(y))
=\frac{m+3}{2},
\end{equation}

Then the maximal radio label of the third type of $P^{\prime}_{3}(j)$ is obtained as
\begin{equation}
    \begin{split}
        f(V_{\frac{m+1}{2}+m(m-1)}(y))&\geq f(V_{m+m(m-1)}(z))+diam(G)+1\\&\quad-d(V_{m+m(m-1)}(z),V_{\frac{m+1}{2}+m(m-1)}(y))\\
&=2m+\frac{m-1}{2}.
    \end{split}
\end{equation}
\end{proof}
\begin{corollary}
    For several paths in $P^{\prime}(j)$, where $j\in[1,n+1]$, then the sum of all radio labelings of the center vertex of $f$ on $P^{\prime}(j)$ is $span(f_3)=\frac{5mn+5m-n+5}{2}$.
\end{corollary}
\begin{proof}
    From the result of Theorem 12, we can clearly obtain the sum of the radio labels of the first and second types of paths in case 1 and case 2 as
    \begin{equation}
        \begin{split}
            span(f_1)&=3\times(2m+\frac{m+1}{2})\\
&=6m+\frac{3m+3}{2}.
        \end{split}
    \end{equation}

    According to case 3, the total span of the radio label of the third type path is denoted as
    \begin{equation}
        \begin{split}
            span(f_2)&=(n-2)\times(2m+\frac{m-1}{2})\\
&=\frac{5mn-10m-n+2}{2}.
        \end{split}
    \end{equation}

    Without loss of generality, the sum of all radio labelings of the central vertex of $f$ on $P^{\prime}(j)$ is
    \begin{equation}
        \begin{split}
            span(f_3)&=span(f_1)+span(f_2)\\
&=\frac{5mn+5m-n+5}{2}.
        \end{split}
    \end{equation}
\end{proof}
\begin{theorem}
    Let $G^{\prime}(\ast\ast)$ be a subgraph on $G$ induced by all vertices and central vertices of $P^{\prime}(j)$, namely $t(1+m(m-1))$, $t(\frac{m+1}{2}+m(m-1))$ and $t(m+m(m-1))$, and $t (m+m(m-1))$. Then $rn(G^{\prime}(\ast\ast)\geq\frac{6mn+5m+5}{2}$.
\end{theorem}
\begin{proof}
     Let $V_1$ and $V_2$ be the central vertices on $t(1+m(m-1))$ and $t(m+m(m-1))$, respectively. There are two vertices $V_p,V_q\in t(\frac{m+1}{2}+m(m-1))$, where $p\neq q$ is not the central vertex. So that $d(V_1,V_p)=d(V_2,V_q)=\frac{m+1}{2}$. There exists a subset $O={X_w}$ in $t(1+m(m-1))$ or $t(m+m(m-1))$ such that $|O|=n$ and a subset $K={Y_s}$ in $t(\frac{m+1}{2}+m(m-1))$. Such that $|K|=n$ and $w\neq s$, then $d(X_w,Y_s)=\frac{m+3}{2}$. Now for all vertex pairs ${X_w,Y_s}$, the sum of the spans of $f$ is denoted $span(f_4)=\frac{mn+n}{2}$. Therefore, the radio number of $G^{\prime}(\ast\ast)$ in the product graph $G$ is
     \begin{equation}
         \begin{split}
             rn(G^{\prime}(\ast\ast))&\geq span(f_3)+span(f_4)\\
&=\frac{6mn+5m+5}{2}.
         \end{split}
     \end{equation}
\end{proof}
\begin{corollary}
    Let $G^{\prime}(i)$ be a subgraph of $G$, where $i\in[1,m]$, then $rn(G^{\prime}(i))\geq\frac{3mn-n+2}{2}$.
\end{corollary}
\begin{proof}
    Let $x_1$ and $y_1$ be the central vertices of $t(d+m(m-1))$ and $t(d+\frac{m+1}{2}+m(m-1))$, respectively. And the subgraph formed by $t(d+m(m-1))$ and $t(d+\frac{m+1}{2}+m(m-1))$ is denoted $G^{\prime}(i)$, where $d\in[2,\frac{m-1}{2}]$ and any two vertices $x_p$ and $y_q$ in $G^{\prime}(i)$. Then the shortest distance between vertices $x_p$ and $y_q$ is
    \begin{equation}
	d(x_p,y_q)=\left\{
	\begin{aligned}
		&\frac{m-1}{2},if \,\, p,q=1;\\
		&\frac{m+1}{2},if \,\, p=1\,or\,q=1;\\
		&\frac{m+3}{2},if \,\, p,q\neq 1.
	\end{aligned}
	\right.
\end{equation}

According to the distance in equation (44), $d(x_1,y_q)=d(y_1,x_p)=\frac{m+1}{2}$. Without loss of generality, let $f_{min}=f(x_1)=0$. Since $d(x_1,y_q)=\frac{m+1}{2}$, $q\in[2,n]$. Therefore, $f(y_q)\geq f(x_1)+diam(G)+1-d(x_1,y_q)$.

When $q=2$, $x_1$ is a central vertex of the subgraph $G^{\prime}(i)$. According to the equation (44) can be obtained
\begin{equation}
    d(x_1,y_2)=\frac{m+1}{2},
\end{equation}

According to the definition of radio label, the multilevel distance label of vertex $y_2$ can be obtained as
\begin{equation}
    \begin{split}
       f(y_2)&\geq f(x_1)+diam(G)+1-d(x_1,y_2)\\
&=m+\frac{m+1}{2}.
    \end{split}
\end{equation}

When $p=3$, the vertices $y_2$ and $x_3$ are leaf vertices of the subgraph $G^{\prime}(i)$. According to the equation (44) can be obtained
\begin{equation}
    d(y_2,x_3)=\frac{m+3}{2},
\end{equation}

Then the radio label of the vertex $x_3$ is
\begin{equation}
    \begin{split}
        f(x_3)&\geq f(y_2)+diam(G)+1-d(y_2,x_3)\\
&=2m+\frac{m+1}{2}-\frac{m-1}{2}.
    \end{split}
\end{equation}

When $q=4$, the vertices $x_3$ and $y_4$ are not the central vertices of the subgraph $G^{\prime}(i)$. According to the equation (44) can be obtained
\begin{equation}
    d(x_3,y_4)=\frac{m+3}{2},
\end{equation}

Then the multilevel distance label of vertex $y_4$ is obtained as
\begin{equation}
    \begin{split}
        f(y_4)&\geq f(x_3)+diam(G)+1-d(x_3,y_4)\\
&=3m+\frac{m+1}{2}-2\times\frac{m-1}{2}.
    \end{split}
\end{equation}

After repeated iterations, when $q=n+1$, the vertices $x_n$ and $y_{n+1}$ are not the central vertices of the subgraph $G^{\prime}(i)$. According to the equation (44) can be obtained
\begin{equation}
    d(x_n,y_{n+1})=\frac{m+3}{2},
\end{equation}

Then the maximal radio label of the subgraph $G^{\prime}(i)$ is
\begin{equation}
    \begin{split}
        f(y_{n+1})&\geq f(x_n)+diam(G)+1-d(x_n,y_{n+1})\\
&=mn+\frac{m+1}{2}-(n-1)\times\frac{m-1}{2}\\
&=\frac{3mn-n+2}{2}.
    \end{split}
\end{equation}
\end{proof}
\begin{theorem}
    Let $G=P(m,m)\Box K_{1,n}$ be the Cartesian product of an $m$-order square network and a star of $n+1$ vertices, where $m\geq 2$ is odd. Take any subgraph $G^{\prime}(\ast\ast\ast)\subset G$. Then $rn(G^{\prime}(\ast\ast\ast))\geq\frac{3m^2n-10mn+4m+3n}{2}$.
\end{theorem}
\begin{proof}
    From Corollary 15, the span of the radio number of $f$ in $G^{\prime}(i)$ is $\frac{3mn-n+2}{2}$, where $i\in[1,m]$. For $G^{\prime}(i)$, $f_{max}=\frac{3mn-n+2}{2}$. Let $d=2$ and $x_1=y_{2+m(m-1)}(n)\in t(2+m(m-1))$, $y^{\prime}_{1}=y_{2+\frac{m+1}{2}+m(m-1)}(n)$ is the center vertex of $t(2+m(m-1))$ and $t(2+\frac{m+1}{2}+m(m-1))$, $d(x_1,y^{\prime}_{1})=\frac{m-1}{2}$. Then 
    \begin{equation}
        \begin{split}
            f(y^{\prime}_{1})&\geq f(x_1)+diam(G)+1-d(x_1,y^{\prime}_(1))\\
&=m+\frac{m+3}{2}+f(x_1).
        \end{split}
    \end{equation}

    This means that for the subgraph $G^{\prime}(3)$ induced by $t(3+m(m-1))$ and $t(3+\frac{m+1}{2}+m(m-1))$, $f_{min}=m+\frac{m+3}{2}+f(x_{1})$ and $f_{max}=f(y^{\prime\prime}_{1})$, where $y^{\prime\prime}_{1}$ is the central vertex of $t(3+\frac{m+1}{2}+m(m-1))$. From the topology of the graph, it can be seen that the graph contains $\frac{m-3}{2}$ $G^{\prime}(i)$ subgraphs. Therefore, the $f_{max}$ of $G^{\prime}(\ast\ast\ast)$ is $f(y^k_{1})\in t(\frac{m-1}{2}+m(m-1))$, where $y^k_{1}$ is the central vertex of $t(\frac{m-1}{2}+m(m-1))$. After continuous iteration, we obtain
    \begin{equation}
        \begin{split}
            f({y^k}_{1})&\geq\frac{m-3}{2}\times\frac{3mn-n+2}{2}+2\times\frac{m+3}{2}\\
&=2m-5mn+\frac{3m^2n+3n}{2}.
        \end{split}
    \end{equation}

    Therefore, we get that the radio number of $G^{\prime}(\ast\ast\ast\ast)$ is $rn(G^{\prime}(\ast\ast\ast))\geq 2m-5mn+\frac{3m^2n+3n}{2}$.
\end{proof}
\begin{theorem}
   Let $G^{\prime\prime}(\ast)$ be a subgraph of $P(m,m)\Box K_{1,n}$, where $m\geq 2$ is odd. Then $rn(G^{\prime\prime}(\ast))\geq\frac{3m^2n-4mn+12m+3n+8}{2}$.
\end{theorem}
\begin{proof}
     From Theorems 14 and 15, we know that $rn(G^{\prime}(\ast\ast)\geq\frac{6mn+5m+5}{2}$ and $rn(G^{\prime}(\ast\ast\ast))\geq\frac{3m^2n-10mn+4m+3n}{2}$. From $G^{\prime\prime}(\ast)=G^{\prime}(\ast\ast)\cup G^{\prime}(\ast\ast\ast)$. Now let $x_1$ be the central vertex of $t(\frac{m-1}{2}+m(m-1))$, and $y_{1}\in t(m+m(m-1))$ and be the central vertex, obviously $d(x_1,y_1)=\frac{m-1}{2}$, take $f(x_1)=f_{max}$. In $G^{\prime}(\ast\ast\ast) $, $f(x_1)=\frac{3m^2n-10mn+4m+3n}{2}$, then
     \begin{equation}
         \begin{split}
             f(y_1)&\geq f(x_1)+diam(G)+1-d(x_1,y_1)\\
&=\frac{3m^3n-10mn+7m+3n+3}{2}.
         \end{split}
     \end{equation}

     For $G^{\prime\prime}(\ast)$, $let f(y_1)=f_{min}$, then
     \begin{equation}
         \begin{split}
             rn(G^{\prime\prime}(\ast))&\geq f(y_1)+G^{\prime}(\ast\ast)\\
&=\frac{3m^3n-10mn+7m+3n+3}{2}+\frac{6mn+5m+5}{2}\\
&=\frac{3m^2n-4mn+12m+3n+8}{2}.
         \end{split}
     \end{equation}
\end{proof}
\begin{theorem}
    Let $G=P(m,m)\Box K_{1,n}$ be the Cartesian product of a square mesh network of order $m$ and a star with $n+1$ vertices, where $m\geq 2$ is odd. Then $rn(G)\geq 5m+4+m^2+\frac{3m^3n+2m^2n+6n-7mn}{4}$.
\end{theorem}
\begin{proof}
From the results of Theorem 9 and Theorem 17, $rn(G^{\prime}(\ast))\geq m^2-m^2n-m+\frac{3m^3n+mn}{4}$ and $rn(G^{\prime\prime}(\ast))\geq\frac{3m^2n-4mn+12m+3n+8}{2}$. When $m$ is odd, the product graph $G$ consists of two parts, namely $G=G^{\prime}(\ast)\cup G^{\prime\prime}(\ast)$. Therefore, the radio number of $G$ is
\begin{equation}
    \begin{split}
        rn(G)&\geq rn(G^{\prime}(\ast))+rn(G^{\prime\prime}(\ast))\\
&=m^2-m^2n-m+\frac{3m^3n+mn}{4}+\frac{3m^2n-4mn+12m+3n+8}{2}\\
&=5m+4+m^2+\frac{3m^3n+2m^2n+6n-7mn}{4}.
    \end{split}
\end{equation}
   \end{proof}

   In summary, we get the main results of this paper. When the Cartesian product of an $m$-order square network and a star with $n+1$ vertices, where $m\geq 2$, its lower bound is
   \begin{equation}
           rn(G)=\left\{
	\begin{aligned}
		&\frac{3m^3}{4}+\frac{2m^2+m^3n+m^2n}{2}, if \,\,\, m \,\,  is \,\,\, even;\\
		&5m+4+m^2+\frac{3m^3n+2m^2n+6n-7mn}{4},if \,\,\, m \,\, is \,\,\, odd.
	\end{aligned}
	\right.
\end{equation}
     \section{Application and numerical simulation}
     Wireless communication networks are an indispensable part of modern society, they provide people with a convenient way of communication. With the rapid development of wireless communication networks, theirs application scenarios are also expanding, such as smart home, telemedicine, driverless cars, virtual reality (VR) and augmented reality (AR), environmental monitoring and so on. Wireless communication networks will continue to develop and integrate with other technologies to bring more convenience to people's lives. Channel assignment is in wireless communication networks. In order to reduce channel interference and improve network transmission efficiency, we can use a variety of optimization strategies, such as graph labeling algorithm, maximum independent set algorithm, dynamic channel assignment algorithm, power control algorithm, multiple access technology and spatial multiplexing technology. By transforming channel assignment into graph vertex labeling problem, network performance can be effectively improved and users can have a network experience.
     
Example 3.1 Suppose there are 96 stations with vertices $V=\{V_1,V_2,... V_{96}\}$. The relationship between the stations is the Cartesian product of a 4-order square mesh network $P(m,m)$ and a 6-vertex star $K_{1,5}$. According to the result of Theorem 6, an optimal radio labeling strategy can be obtained, and the radio label of the network structure is 304.

Example 3.2 Suppose there are 150 stations with vertices $V=\{V_1,V_2,... V_{150}\}$. The relationship between the stations is the Cartesian product of a 5-order square mesh network $P(m,m)$ and a 6-vertex star $K_{1,5}$. According to the result of Theorem 17, an optimal radio labeling strategy can be obtained, and the radio label of this network structure is 648.
\begin{figure}[htpb]
	\centerline{\includegraphics{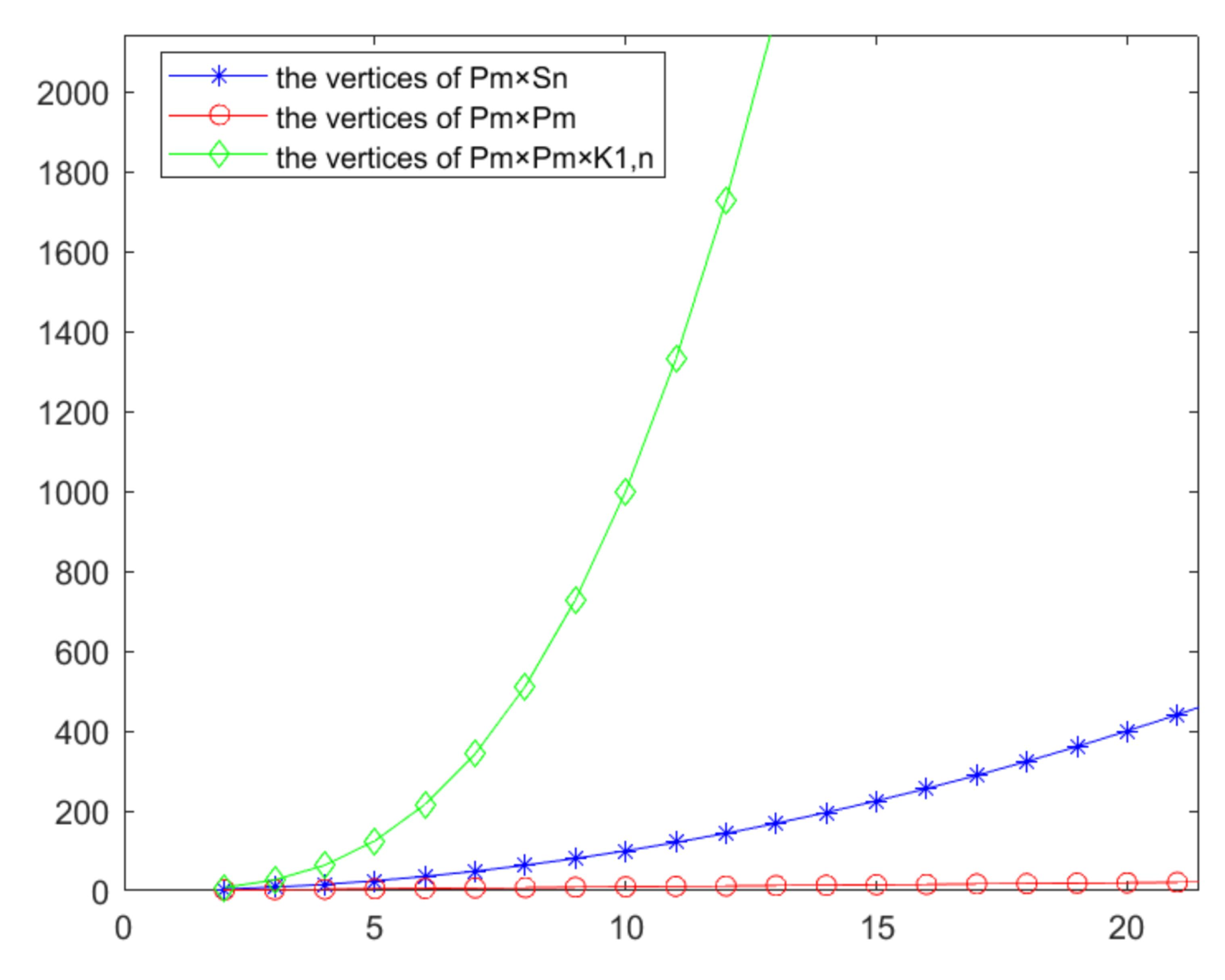}}
	\vspace{-4mm}
	\caption{The number of vertices of the three types of topologies is compared}
	\label{fig}
\end{figure}

By computer simulation, several large networks constructed by Cartesian product method in bus and star network topologies were compared. As shown in Fig 4, the number of sites of three large networks constructed by the Cartesian product method such as square of path \cite{2009S}, star and path \cite{2022B}, and square mesh network and star are compared. It can be found that, in the same order, the model used in this paper constructs a large network with more sites and constructs a larger network topology with a much higher number of sites than the other two graph construction methods, followed by the network topology constructed by the Cartesian product of star and path.
\begin{figure}[htpb]
	\centerline{\includegraphics{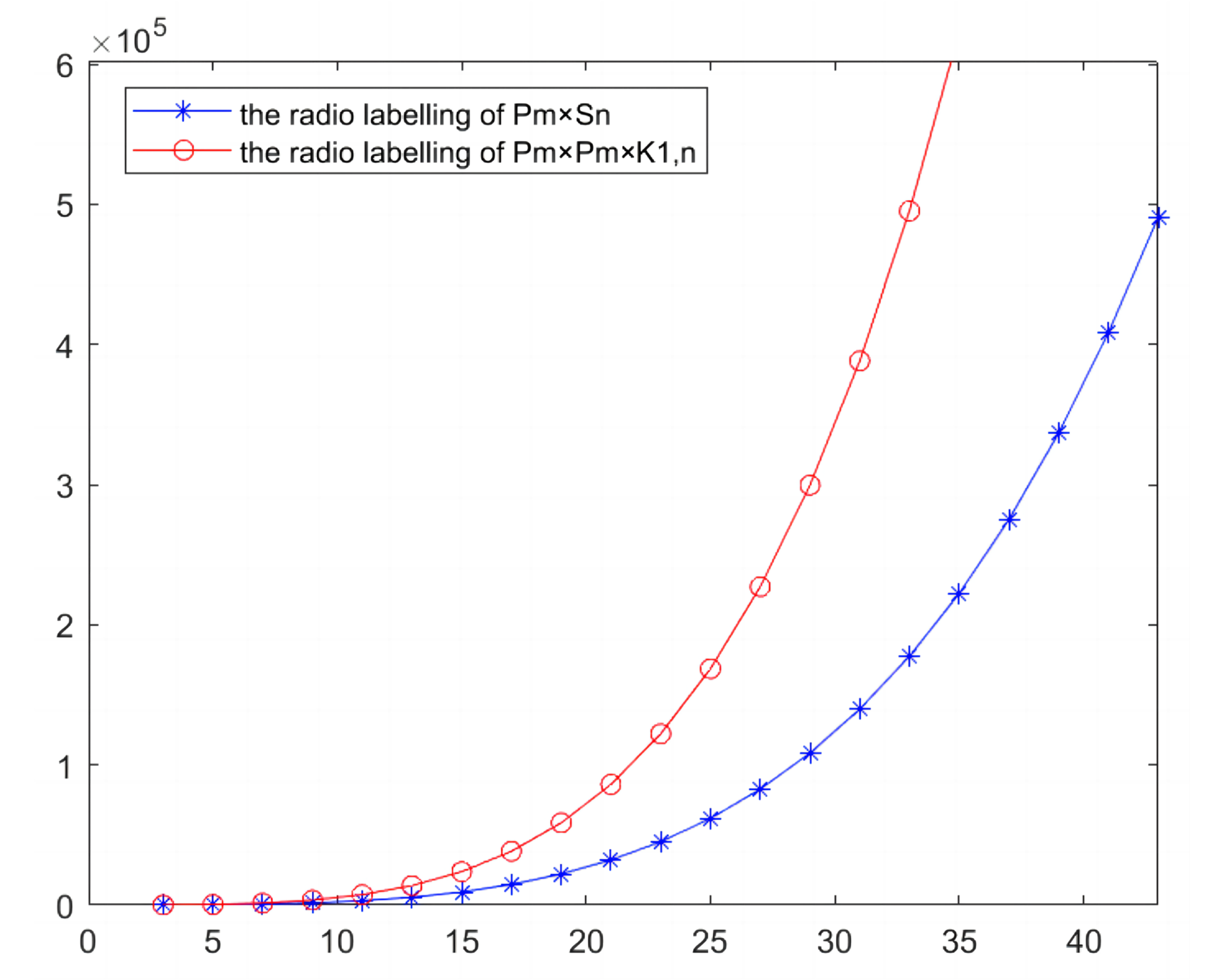}}
	\vspace{-4mm}
	\caption{When the total number of sites is the same, compare the topology of the Cartesian product of square mesh networks and stars with the Cartesian product of stars and paths}
	\label{fig}
\end{figure}
\begin{figure}[htpb]
	\centerline{\includegraphics{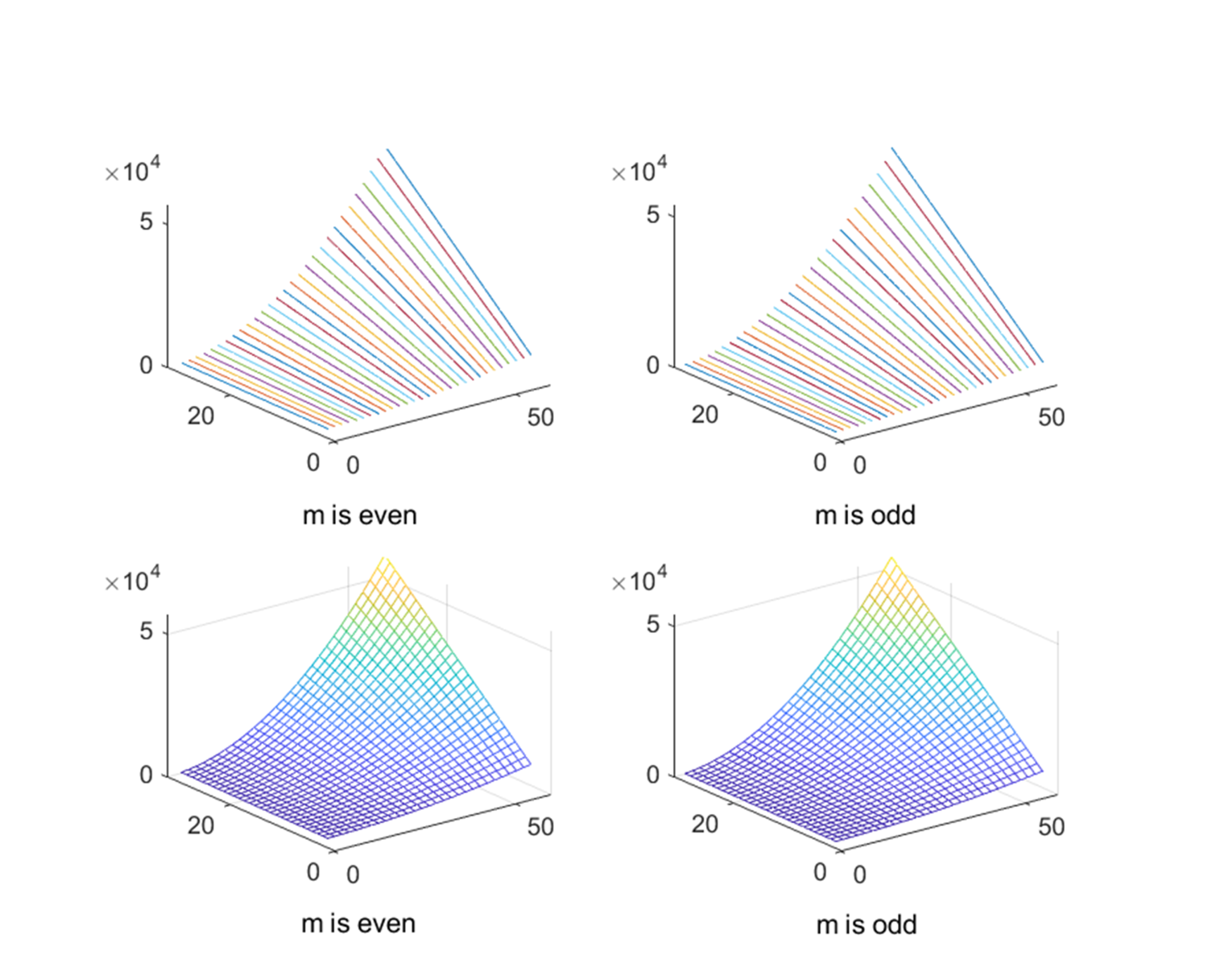}}
	\vspace{-6mm}
	\caption{The Cartesian product  of a star and a path $S_n\Box P_m$, where the left column represents the number of radios in the Cartesian product when $m$ is even, and the right column represents the number of radios in the Cartesian product when $m$ is odd}
	\label{fig}
\end{figure}
\begin{figure}[htpb]
	\centerline{\includegraphics{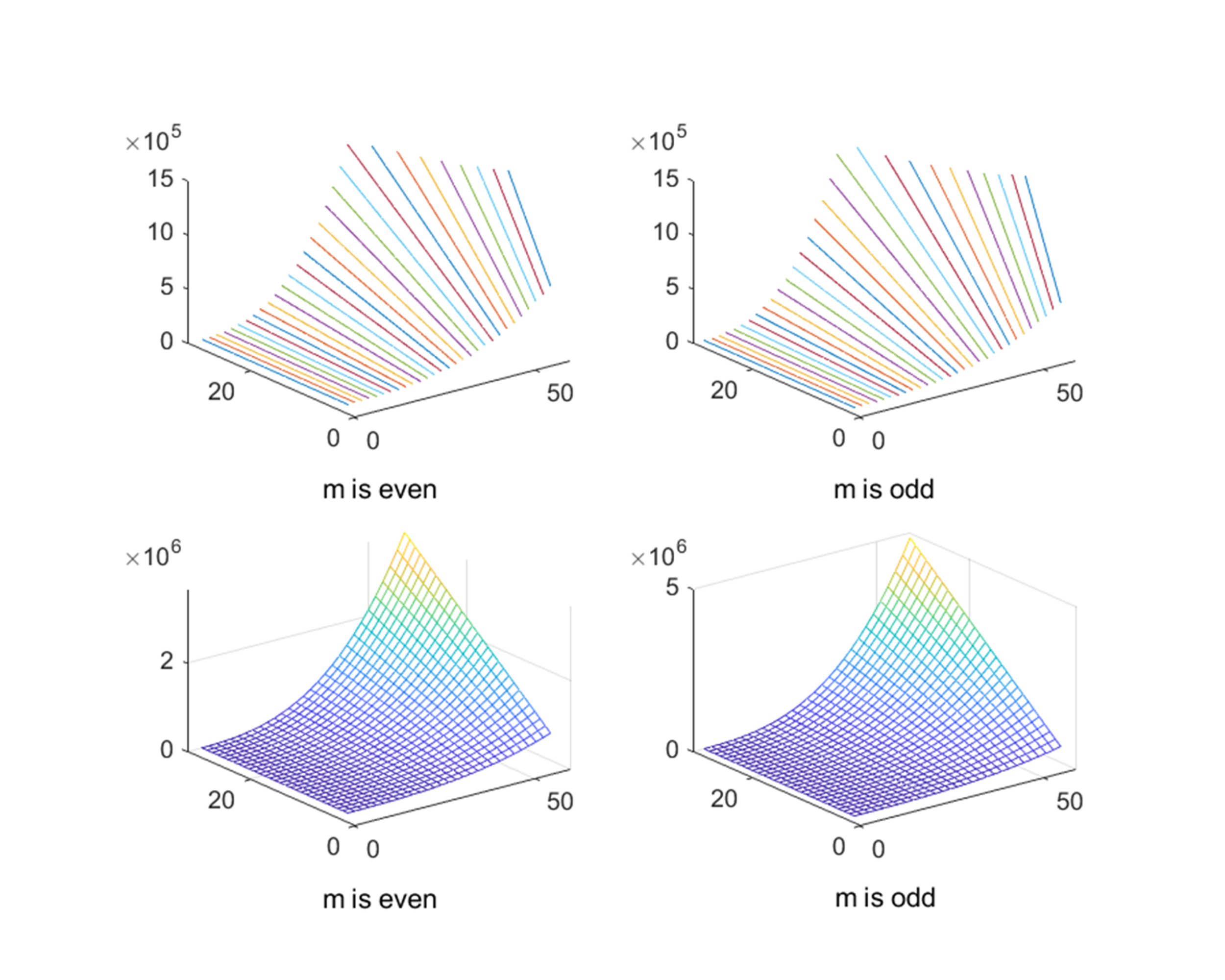}}
	\vspace{-6mm}
	\caption{The Cartesian product of a star and a path $P(m,m)\Box K_{1,n}$, where the left column represents the number of radios in the Cartesian product when $m$ is even, and the right column represents the number of radios in the Cartesian product diagram when $m$ is odd}
	\label{fig}
\end{figure}

The Fig 4 illustrates the comparison of the number of vertices for different network topologies and refers to the results of literature \cite{2009S,2022B}. We can find that the network topology constructed by Cartesian product of star and square mesh network has $m$ times more sites than the number of sites of the network constructed by star and path. When the number of sites of the network topology is the same, as shown in Fig 5, the channel usage of the network constructed by the Cartesian product of star and square mesh network is significantly lower than that of the large network constructed by the Cartesian product of star and square mesh network.

In order to compare in more depth the channel usage of the two topologies, the Cartesian product of a star with a path and the Cartesian product of a star and a square mesh network, we have compared and analyzed Fig 6 and 7 by computer simulation. Although the topology of the Cartesian product graph of a star and a square mesh network uses slightly more channels than the topology of the Cartesian product graph of a star and a path, the network structure of the former is more complex, and the number of interactions between sites far exceeds that of the latter.
To summarize, the network topology constructed by the Cartesian product graph of star and square mesh network is not only able to construct large-scale networks, but also more suitable to be applied to the construction of inter-city communication networks. 
\section{Conclusion}
This paper primarily uses the Cartesian product construction method to study the Cartesian product graph of star and square mesh networks, obtaining the optimal radio label and its lower bound. We also analyze the network constructed by different Cartesian product structures through computer simulation. However, it is insufficient to only study the square mesh network. Future work will further investigate the radio labels of networks with other structures and the product graphs of stars.
\section*{Acknowledgment}
This work was supported in part by the National Natural Science Foundation of China under Grant 11551002, in part by the Natural Science Foundation of Qinghai Province under Grant 2019-ZJ-7093.




\bibliographystyle{elsarticle-num-names} 
\bibliography{main.bib}

\begin{thebibliography}{29}
\expandafter\ifx\csname natexlab\endcsname\relax\def\natexlab#1{#1}\fi
\providecommand{\url}[1]{\texttt{#1}}
\providecommand{\href}[2]{#2}
\providecommand{\path}[1]{#1}
\providecommand{\DOIprefix}{doi:}
\providecommand{\ArXivprefix}{arXiv:}
\providecommand{\URLprefix}{URL: }
\providecommand{\Pubmedprefix}{pmid:}
\providecommand{\doi}[1]{\href{http://dx.doi.org/#1}{\path{#1}}}
\providecommand{\Pubmed}[1]{\href{pmid:#1}{\path{#1}}}
\providecommand{\bibinfo}[2]{#2}
\ifx\xfnm\relax \def\xfnm[#1]{\unskip,\space#1}\fi
\bibitem[{Ming(2007)}]{2007X}
\bibinfo{author}{X.~J. Ming}, \bibinfo{title}{Combinatorial network theory}, \bibinfo{publisher}{Beijing Science Press}, \bibinfo{year}{2007}.
\bibitem[{Sabidussi(1959)}]{1959G}
\bibinfo{author}{G.~Sabidussi},
\newblock \bibinfo{title}{Graph multiplication},
\newblock \bibinfo{journal}{Mathematische Zeitschrift} \bibinfo{volume}{72} (\bibinfo{year}{1959}) \bibinfo{pages}{446--457}. \URLprefix \url{https://api.semanticscholar.org/CorpusID:186238657}.
\bibitem[{Hale(1980)}]{1980F}
\bibinfo{author}{W.~K. Hale},
\newblock \bibinfo{title}{Frequency assignment: Theory and applications},
\newblock \bibinfo{journal}{Proceedings of the IEEE} \bibinfo{volume}{68} (\bibinfo{year}{1980}) \bibinfo{pages}{1497--1514}. \URLprefix \url{https://api.semanticscholar.org/CorpusID:37268885}.
\bibitem[{Griggs and Yeh(1992)}]{1992G}
\bibinfo{author}{J.~R. Griggs}, \bibinfo{author}{R.~K. Yeh},
\newblock \bibinfo{title}{Labelling graphs with a condition at distance 2},
\newblock \bibinfo{journal}{SIAM Journal on Discrete Mathematics} \bibinfo{volume}{5} (\bibinfo{year}{1992}) \bibinfo{pages}{586--595}. \URLprefix \url{https://doi.org/10.1137/0405048}. \DOIprefix\doi{10.1137/0405048}.
\bibitem[{Chartrand et~al.(2001)Chartrand, Erwin, Harary, and Zhang}]{2001R}
\bibinfo{author}{G.~Chartrand}, \bibinfo{author}{D.~J. Erwin}, \bibinfo{author}{F.~Harary}, \bibinfo{author}{P.~Zhang},
\newblock \bibinfo{title}{Radio labelings of graphs},
\newblock volume~\bibinfo{volume}{33}, \bibinfo{year}{2001}, pp. \bibinfo{pages}{77--85}.
\bibitem[{Haider et~al.(2021)Haider, Koam, and Ahmad}]{2021R}
\bibinfo{author}{A.~Haider}, \bibinfo{author}{A.~N.~A. Koam}, \bibinfo{author}{A.~Ahmad},
\newblock \bibinfo{title}{Radio labeling associated with a class of commutative rings using zero-divisor graph},
\newblock \bibinfo{journal}{Intelligent Automation \& Soft Computing} \bibinfo{volume}{30} (\bibinfo{year}{2021}) \bibinfo{pages}{787--794}. \URLprefix \url{https://api.semanticscholar.org/CorpusID:238869185}.
\bibitem[{Badr and Moussa(2020)}]{2020U}
\bibinfo{author}{E.~M. Badr}, \bibinfo{author}{M.~I. Moussa},
\newblock \bibinfo{title}{An upper bound of radio k-coloring problem and its integer linear programming model},
\newblock \bibinfo{journal}{Wireless Networks} \bibinfo{volume}{26} (\bibinfo{year}{2020}) \bibinfo{pages}{4955--4964}. \URLprefix \url{https://api.semanticscholar.org/CorpusID:88487072}.
\bibitem[{Kim and Song(2015)}]{2015R}
\bibinfo{author}{W.~Kim}, \bibinfo{author}{B.~C. Song},
\newblock \bibinfo{title}{Radio number for the product of a path and a complete graph},
\newblock \bibinfo{journal}{Journal of combinatorial optimization} \bibinfo{volume}{30} (\bibinfo{year}{2015}) \bibinfo{pages}{139--149}. \URLprefix \url{https://api.semanticscholar.org/CorpusID:29444285}.
\bibitem[{Palani and Subi(2022)}]{2022r}
\bibinfo{author}{K.~Palani}, \bibinfo{author}{S.~S. Subi},
\newblock \bibinfo{title}{Radio labeling of some splitting graphs},
\newblock \bibinfo{journal}{Advances and Applications in Mathematical Sciences} \bibinfo{volume}{21} (\bibinfo{year}{2022}) \bibinfo{pages}{1217--1228}.
\bibitem[{Bantva and Liu(2021)}]{2021O}
\bibinfo{author}{D.~Bantva}, \bibinfo{author}{D.~D.-F. Liu},
\newblock \bibinfo{title}{Optimal radio labellings of block graphs and line graphs of trees},
\newblock \bibinfo{journal}{Theoretical Computer Science} \bibinfo{volume}{891} (\bibinfo{year}{2021}) \bibinfo{pages}{90--104}. \URLprefix \url{https://api.semanticscholar.org/CorpusID:237353481}.
\bibitem[{Bantva and Liu(2024)}]{2024R}
\bibinfo{author}{D.~Bantva}, \bibinfo{author}{D.~D.-F. Liu},
\newblock \bibinfo{title}{Radio number for the cartesian product of two trees},
\newblock \bibinfo{journal}{Discrete Applied Mathematics} \bibinfo{volume}{342} (\bibinfo{year}{2024}) \bibinfo{pages}{304--316}. \URLprefix \url{https://api.semanticscholar.org/CorpusID:247158641}.
\bibitem[{Korže et~al.(2022)Korže, Shao, and Vesel}]{2022N}
\bibinfo{author}{D.~Korže}, \bibinfo{author}{Z.~Shao}, \bibinfo{author}{A.~Vesel},
\newblock \bibinfo{title}{New results on radio k-labelings of distance graphs},
\newblock \bibinfo{journal}{Discrete Applied Mathematics} \bibinfo{volume}{319} (\bibinfo{year}{2022}) \bibinfo{pages}{472--479}. \URLprefix \url{https://api.semanticscholar.org/CorpusID:244232278}.
\bibitem[{Bloomfield and Ramirez(2022)}]{2021C}
\bibinfo{author}{D.~D.-F. Bloomfield, Colin~andLiu}, \bibinfo{author}{J.~Ramirez},
\newblock \bibinfo{title}{Radio-k-labeling of cycles for large k},
\newblock \bibinfo{journal}{Discrete Applied Mathematics} \bibinfo{volume}{316} (\bibinfo{year}{2022}) \bibinfo{pages}{60--70}. \URLprefix \url{https://api.semanticscholar.org/CorpusID:235669876}.
\bibitem[{Liu et~al.(2021)Liu, Saha, and Das}]{2021I}
\bibinfo{author}{D.~D.-F. Liu}, \bibinfo{author}{L.~Saha}, \bibinfo{author}{S.~Das},
\newblock \bibinfo{title}{Improved lower bounds for the radio number of trees},
\newblock \bibinfo{journal}{Theoretical Computer Science} \bibinfo{volume}{851} (\bibinfo{year}{2021}) \bibinfo{pages}{1--13}. \URLprefix \url{https://api.semanticscholar.org/CorpusID:219751604}.
\bibitem[{Mari and Jeyaraj(2023)}]{2023L}
\bibinfo{author}{B.~Mari}, \bibinfo{author}{R.~S. Jeyaraj},
\newblock \bibinfo{title}{Radio labeling of supersub—division of path graphs},
\newblock \bibinfo{journal}{IEEE Access} \bibinfo{volume}{11} (\bibinfo{year}{2023}) \bibinfo{pages}{123096--123103}. \URLprefix \url{https://api.semanticscholar.org/CorpusID:264549402}.
\bibitem[{Niranjan and Kola(2020)}]{2020O}
\bibinfo{author}{P.~K. Niranjan}, \bibinfo{author}{S.~R. Kola},
\newblock \bibinfo{title}{On the radio number for corona of paths and cycles},
\newblock \bibinfo{journal}{AKCE International Gournal of Graphs and Combinatorics} \bibinfo{volume}{17} (\bibinfo{year}{2020}) \bibinfo{pages}{269--275}. \URLprefix \url{https://api.semanticscholar.org/CorpusID:198469408}.
\bibitem[{Araki(2009)}]{2009P}
\bibinfo{author}{T.~Araki},
\newblock \bibinfo{title}{Labeling bipartite permutation graphs with a condition at distance two},
\newblock \bibinfo{journal}{Discrete Applied Mathematics} \bibinfo{volume}{157} (\bibinfo{year}{2009}) \bibinfo{pages}{1677--1686}. \URLprefix \url{https://api.semanticscholar.org/CorpusID:2486778}.
\bibitem[{Liu and Zhu(2005)}]{2005M}
\bibinfo{author}{D.~D.-F. Liu}, \bibinfo{author}{X.~D. Zhu},
\newblock \bibinfo{title}{Multilevel distance labelings for paths and cycles},
\newblock \bibinfo{journal}{SIAM Journal on discrete mathematics} \bibinfo{volume}{19} (\bibinfo{year}{2005}) \bibinfo{pages}{610--621}. \URLprefix \url{https://api.semanticscholar.org/CorpusID:1708763}.
\bibitem[{Bonomo and Cerioli(2011)}]{2011T}
\bibinfo{author}{F.~Bonomo}, \bibinfo{author}{M.~R. Cerioli},
\newblock \bibinfo{title}{On the l(2, 1)-labelling of block graphs},
\newblock \bibinfo{journal}{International Journal of Computer Mathematics} \bibinfo{volume}{88} (\bibinfo{year}{2011}) \bibinfo{pages}{468--475}. \URLprefix \url{https://api.semanticscholar.org/CorpusID:768797}.
\bibitem[{Yue and Li(2024)}]{2024Y}
\bibinfo{author}{Y.~X. Yue}, \bibinfo{author}{F.~Li},
\newblock \bibinfo{title}{Fault diameter of strong product graph of an arbitrary connected graph and a complete graph},
\newblock \bibinfo{journal}{Engineering Letters} \bibinfo{volume}{32} (\bibinfo{year}{2024}).
\bibitem[{Wei and Li(2024)}]{2024W}
\bibinfo{author}{L.~Y. Wei}, \bibinfo{author}{F.~Li},
\newblock \bibinfo{title}{Italian domination number of strong product of cycles},
\newblock \bibinfo{journal}{IAENG International Journal of Applied Mathematics} \bibinfo{volume}{54} (\bibinfo{year}{2024}). \URLprefix \url{https://api.semanticscholar.org/CorpusID:269139001}.
\bibitem[{Hong and Li(2023)}]{2023H}
\bibinfo{author}{J.~J. Hong}, \bibinfo{author}{F.~Li},
\newblock \bibinfo{title}{Optimal radio labeling for the strong product of multiple paths},
\newblock in: \bibinfo{booktitle}{2023 IEEE 6th International Conference on Computer and Communication Engineering Technology}, \bibinfo{organization}{IEEE}, \bibinfo{year}{2023}, pp. \bibinfo{pages}{167--172}. \URLprefix \url{https://api.semanticscholar.org/CorpusID:265737168}.
\bibitem[{Li et~al.(2011)Li, Wang, Xu, and Zhao}]{2011S}
\bibinfo{author}{F.~Li}, \bibinfo{author}{W.~Wang}, \bibinfo{author}{Z.~B. Xu}, \bibinfo{author}{H.~X. Zhao},
\newblock \bibinfo{title}{Some results on the lexicographic product of vertex-transitive graphs},
\newblock \bibinfo{journal}{Applied Mathematics Letters} \bibinfo{volume}{24} (\bibinfo{year}{2011}) \bibinfo{pages}{1924--1926}. \URLprefix \url{https://api.semanticscholar.org/CorpusID:6444508}.
\bibitem[{Li et~al.(2012)Li, Xu, Zhao, and Wang}]{2012L}
\bibinfo{author}{F.~Li}, \bibinfo{author}{Z.~B. Xu}, \bibinfo{author}{H.~X. Zhao}, \bibinfo{author}{W.~Wang},
\newblock \bibinfo{title}{On the number of spanning trees of the lexicographic product of networks},
\newblock \bibinfo{journal}{Scientia Sinica Informationis} \bibinfo{volume}{42} (\bibinfo{year}{2012}) \bibinfo{pages}{949--959}. \URLprefix \url{https://api.semanticscholar.org/CorpusID:124155128}.
\bibitem[{Badr et~al.(2022)Badr, Nada, Ali Al-Shamiri, Abdel-Hay, and Elrokh}]{2022A}
\bibinfo{author}{E.~Badr}, \bibinfo{author}{S.~Nada}, \bibinfo{author}{M.~M. Ali Al-Shamiri}, \bibinfo{author}{A.~Abdel-Hay}, \bibinfo{author}{A.~Elrokh},
\newblock \bibinfo{title}{A novel mathematical model for radio mean square labeling problem},
\newblock \bibinfo{journal}{Journal of Mathematics}  (\bibinfo{year}{2022}). \URLprefix \url{https://api.semanticscholar.org/CorpusID:245915827}.
\bibitem[{Zhang et~al.(2019)Zhang, Nazeer, Habib, Zia, and Ren}]{2019Z}
\bibinfo{author}{F.~Zhang}, \bibinfo{author}{S.~Nazeer}, \bibinfo{author}{M.~Habib}, \bibinfo{author}{T.~J. Zia}, \bibinfo{author}{Z.~Ren},
\newblock \bibinfo{title}{Radio number for generalized petersen graphs $ p (n, 2) $},
\newblock \bibinfo{journal}{IEEE Access} \bibinfo{volume}{7} (\bibinfo{year}{2019}) \bibinfo{pages}{142000--142008}. \URLprefix \url{https://api.semanticscholar.org/CorpusID:204076102}.
\bibitem[{Chavez et~al.(2021)Chavez, Liu, and Shurman}]{2021A}
\bibinfo{author}{A.~Chavez}, \bibinfo{author}{D.~D.-F. Liu}, \bibinfo{author}{M.~Shurman},
\newblock \bibinfo{title}{Optimal radio-k-labelings of trees},
\newblock \bibinfo{journal}{European Journal of Combinatorics} \bibinfo{volume}{91} (\bibinfo{year}{2021}) \bibinfo{pages}{1--16}. \URLprefix \url{https://api.semanticscholar.org/CorpusID:222827344}.
\bibitem[{Liu and Xie(2009)}]{2009S}
\bibinfo{author}{D.~D.-F. Liu}, \bibinfo{author}{M.~Xie},
\newblock \bibinfo{title}{Radio number for square paths},
\newblock \bibinfo{journal}{Ars Combinatoria} \bibinfo{volume}{90} (\bibinfo{year}{2009}) \bibinfo{pages}{307--319}. \URLprefix \url{https://api.semanticscholar.org/CorpusID:14498312}.
\bibitem[{Adefokun and Ajayi(2022)}]{2022B}
\bibinfo{author}{T.~C. Adefokun}, \bibinfo{author}{D.~O. Ajayi},
\newblock \bibinfo{title}{Bounds of the radio number of stacked-book graphs with odd paths},
\newblock \bibinfo{journal}{International Journal of Mathematical Combinatorics}  (\bibinfo{year}{2022}) \bibinfo{pages}{1--11}. \URLprefix \url{https://api.semanticscholar.org/CorpusID:252185381}.

\end{thebibliography}


\vspace{10mm}

\begin{wrapfigure}{l}{2cm}
	\centering
	\includegraphics[width=0.15\textwidth]{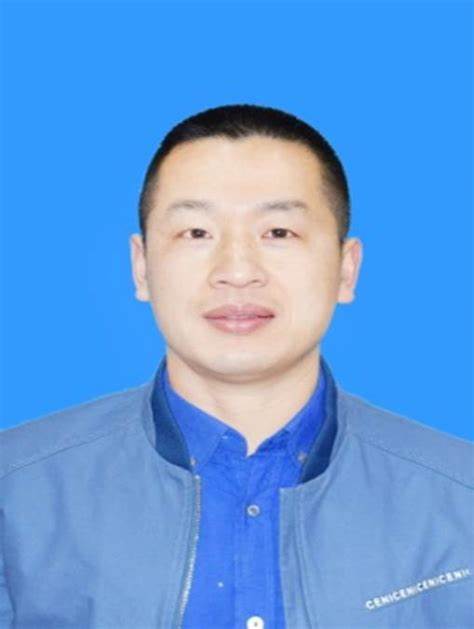}
	\end{wrapfigure}
	\textbf {Feng Li}
	was born in AnHui, China, in 1981. He is currently teaching at the Qinghai Normal University, Professor, Doctoral Supervision. He received his Ph.D. from the School of Mathematics and Statistics, Xi’an Jiaotong University. His research interests include graph theory, analysis and design of reliable combinatorial networks, the optimization theory and algorithms of interconnection networks, and machine learning. He
	published over 90 academic research papers, and he has more than 80	software copyrights in graph theory and combinatorial networks.

	\vspace{10mm}\par

	\begin{wrapfigure}{l}{2cm}
		\centering
		\includegraphics[width=0.15\textwidth]{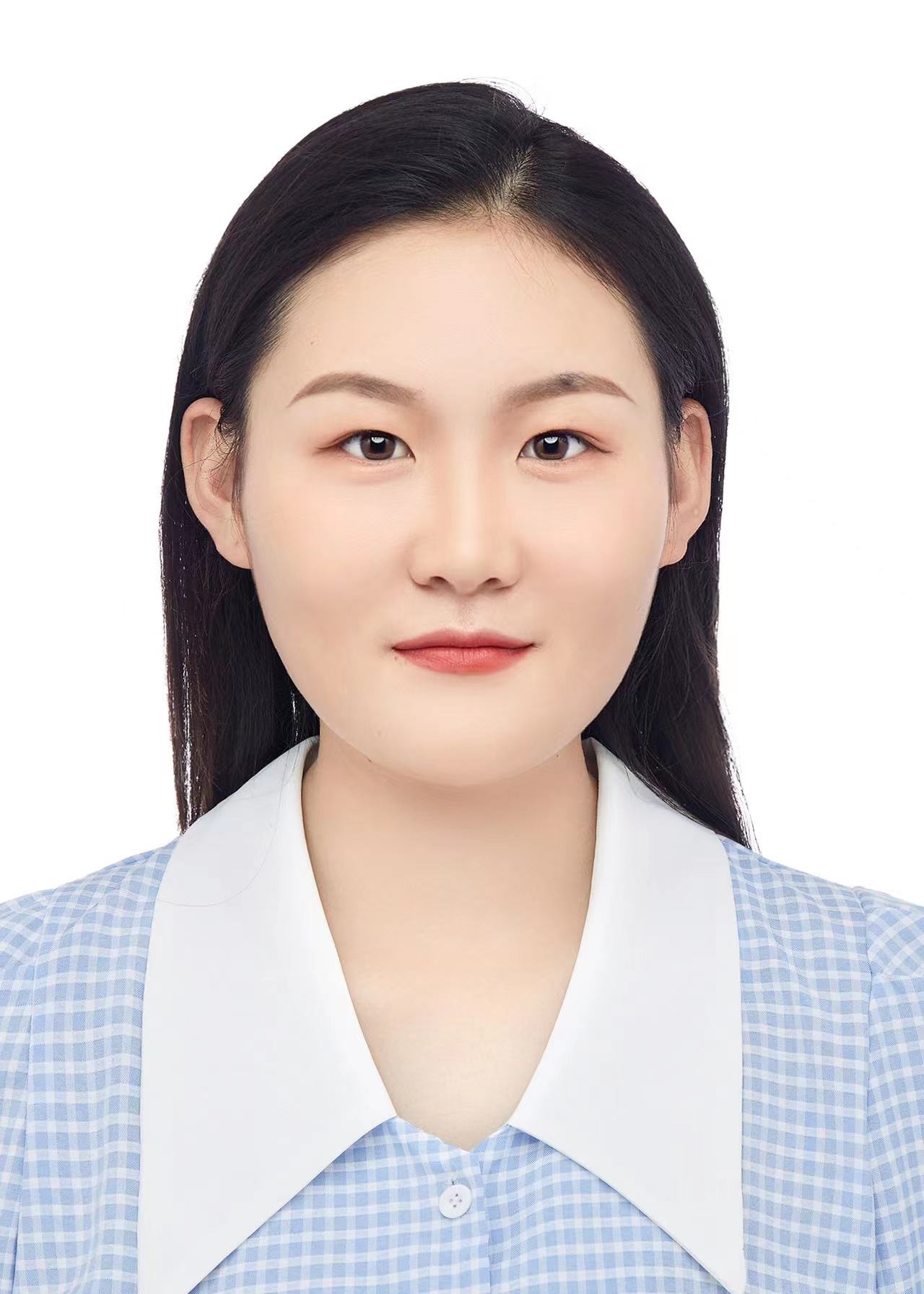}
			
		\end{wrapfigure}
		\textbf {Linlin Cui}
		was born in Henan, China, in 1998. She is currently
		pursuing the M.Sc. in the School of Computer Science, Qinghai Normal University, China. She received her bachelor in 2022 from School of Computer Science, Qinghai Normal University. Her research interests include graph theory, analysis and design of reliable combinatorial	networks.
\end{document}